\documentclass[12pt]{article}

\usepackage{t1enc}
\usepackage[latin1]{inputenc}
\usepackage[english]{babel}
\usepackage{amsmath,amsthm}
\usepackage{amsfonts}
\usepackage{latexsym}
\usepackage{graphicx}
\usepackage{txfonts}
\usepackage[natural]{xcolor}
\usepackage{rotating}
\usepackage{lineno}

\newtheorem{thm}{Theorem}[section]
\newtheorem{lem}[thm]{Lemma}
\newtheorem{cor}[thm]{Corollary}

\makeatletter

\newcommand{\Rmnum}[1]{\expandafter\@slowromancap\romannumeral #1@}
\makeatother

\begin{document}


\title{List edge coloring of outer-1-planar graphs\thanks{Supported by the Natural Science Basic Research Plan in Shaanxi Province of China (No.\,2017JM1010), the Fundamental Research Funds for the Central Universities (No.\,JB170706), and
the National Natural Science Foundation of China (Nos.\,11871055, 11301410).}}
\author{Xin Zhang\thanks{Email address: xzhang@xidian.edu.cn.}\\
{\small School of Mathematics and Statistics, Xidian University, Xi'an 710071, China}}

\maketitle

\begin{abstract}\baselineskip  0.6cm
A graph is outer-1-planar if it can be drawn in the plane so that all vertices are on the outer face and each edge is crossed at most once. It is known that
the list edge chromatic number $\chi'_l(G)$ of any outer-1-planar graph $G$ with maximum degree $\Delta(G)\geq 5$ is exactly its maximum degree. In this paper, we prove $\chi'_l(G)=\Delta(G)$ for outer-1-planar graphs $G$ with $\Delta(G)=4$ and with the crossing distance being at least 3.
\\[.2em]
\end{abstract}

\baselineskip  0.6cm

\section{Introduction}

In this paper, all graphs are finite, simple and undirected. By $V(G)$, $E(G)$,
$\delta(G)$ and $\Delta (G)$, we denote the vertex set, the edge set, the minimum degree and the maximum degree
of a graph $G$, respectively. The \emph{order} $|G|$ of a graph $G$ is $|V(G)|$ and the \emph{size} of $G$ is $|E(G)|$. The \emph{distance} $d_G(u,w)$ between two vertices $u$ and $w$ of a connected graph $G$ is the minimum length of the path (i.e., the number of edges on the path) connecting them.

The problem of coloring a graph arises in many practical areas such as pattern matching, sports scheduling, designing seating plans, exam timetabling, the scheduling of taxis, and solving Sudoku puzzles \cite{Lewis}. There are many kinds of colorings of graphs, and in this paper we mainly focus on the edge coloring. Precisely,
an \emph{edge coloring} of a graph $G$ is an assignment of colors to the edges of $G$ such that every pair of adjacent edges receive different colors. An \emph{edge $k$-coloring} of a graph $G$ is an edge coloring of $G$ from a set of $k$ colors.
The minimum positive integer $k$ for which $G$ has an edge $k$-coloring, denoted by $\chi'(G)$, is the \emph{edge chromatic number} of $G$. The well-known Vizing's Theorem states that $\Delta(G)\leq \chi'(G)\leq \Delta(G)+1$ for any simple graph $G$.

Suppose that a set $L(e)$ of colors, called a \emph{list} of $e$, is assigned to each edge $e\in E(G)$. \emph{$L$-coloring} an edge $e$ means coloring $e$ with a color in $L(e)$.
An \emph{edge $L$-coloring} of $G$ is an edge coloring $c$ so that $c(e)\in L(e)$ for every $e\in E(G)$.
We say that $G$ is \emph{edge $k$-choosable} if $G$ has an edge $L$-coloring whenever $|L(e)|=k$ for every $e\in E(G)$. The minimum integer $k$ for which $G$ is edge $k$-choosable is the \emph{list edge chromatic number} of $G$, denoted by $\chi'_l(G)$.

The most famous open problem concerning list edge coloring is probably the \emph{list edge coloring conjecture} (LECC for short):
$$\chi'_{l}(G)=\chi'(G)$$ for any graph $G$.
This conjecture has a fuzzy origin. Jensen and Toft overview its history in their book \cite{JT-book}.

LECC is regarded as very difficult, and is still widely open. Some partial results were however obtained in the special case of
planar graphs. For example, LECC is true for planar graphs with maximum degree at least 12 \cite{Borodin.lcc}, series-parallel graphs \cite{JMT}, outerplanar graphs \cite{Wang2}, near-outerplanar graphs \cite{HW}, and pseudo-outerplanar (outer-1-planar) graphs with maximum degree at least 5 \cite{TZ}.

A graph is \emph{outer-1-planar} if it can be drawn in the plane so that all vertices are on the outer face and each edge is crossed at most once.
For example, $K_{2,3}$ and $K_4$ are outer-1-planar graphs.
Outer-1-planar graphs were first introduced by Eggleton \cite{Eggleton} who called them \emph{outerplanar graphs with edge crossing number one}, and were also investigated under the notion of \emph{pseudo-outerplanar graphs} by Zhang, Liu and Wu \cite{LTC,ZPOPG}.

 In this paper, we use some notions and notations from \cite{total}. A drawing of an outer-1-planar graph in the plane such that its outer-1-planarity is satisfied is an \emph{outer-1-plane graph}, and we call it \emph{good} if the number of its crossings is as small as possible.
Note that every crossing in an outer-1-plane graph $G$ is generated by two mutually crossed edges, thus every crossing $c$ correspond to a vertex set $M_G(c)$ of size four, where $M_G(c)$ consists of the end-vertices of the two edges that generate $c$.
For crossings $c_1$ and $c_2$ in an outer-1-plane graph $G$, define $d_G(c_1,c_2)=\min\{d_G(v_1,v_2)~|~v_1\in M_G(c_1)~{\rm and}~v_2\in M_G(c_2)\}$ to be the \emph{distance} in the drawing $G$ between $c_1$ and $c_2$.
By $$\vartheta(G)=\min\{d_{G'}(c_1,c_2)~|~G'~ is~a~good~drawing~of~G,~and~c_1,c_2~are~distinct~crossings~of~G'\}$$we denote the \emph{crossing distance} of an outer-1-planar graph $G$. Note that we will set $\vartheta(G)=\infty$ if $G$ has a good drawing with at most one crossing.

For every outer-1-planar graph $G$, it was proved in \cite[Theorem 5.3]{ZPOPG} that $\chi'(G)=\Delta(G)$ if $\Delta(G)\geq 4$, and in \cite[Theorem 2.5]{TZ} that $\chi'_l(G)=\chi'(G)=\Delta(G)$ if $\Delta(G)\geq 5$. In this paper we will prove the following results.

\begin{thm}\label{o1p}
Let $G$ be an outer-1-planar graph. If
$\Delta(G)=4$ and $\vartheta(G)\geq 3$,
then $\chi'_l(G)=\chi'(G)=4$.
\end{thm}

\begin{thm}\label{3}
If $G$ is an outer-1-planar graph with $\Delta(G)=3$, then $\chi'_l(G)\leq 4$.
\end{thm}

The following corollary from Theorem \ref{3} is immediate.

\begin{cor}
If $G$ is an outer-1-planar graph with $\Delta(G)=3$ and $\chi'(G)=4$, then $\chi'_l(G)=\chi'(G)$.
\end{cor}

Note that there exist infinitely many outer-1-planar graph $G$ with $\Delta(G)=3$ and $\chi'(G)=4$, see \cite[Theorem 3.2]{edge}.


\section{Preliminaries}

Throughout this section, $G$ will be a good 2-connected outer-1-plane graph, and by $v_1, \ldots, v_{|G|}$ we denote the vertices of $G$ with clockwise ordering on the boundary.

Let $\mathcal{V}[v_i,v_j]=\{v_i, v_{i+1}, \ldots, v_j\}$ and $\mathcal{V}(v_i,v_j)=\mathcal{V}[v_i,v_j]\backslash \{v_i,v_j\}$, where the subscripts are taken modulo $|G|$. Set $\mathcal{V}[v_i,v_i]=V(G)$ and $\mathcal{V}(v_i,v_i)=V(G)\setminus \{v_i\}$. By $G[v_i,v_j]$ and $G(v_i,v_j)$, we denote the subgraph of $G$ induced by $\mathcal{V}[v_i,v_j]$ and $\mathcal{V}(v_i,v_j)$, respectively.
If there is no edge between $\mathcal{V}(v_i,v_j)$ and $\mathcal{V}(v_j,v_i)$, then $\hat{G}_{i,j}$ denotes the graph obtained from $G[v_i,v_j]$ by adding edge $v_iv_j$ if it does not exist in $G[v_i,v_j]$ (otherwise $\hat{G}_{i,j}$ is $G[v_i,v_j]$ itself). Clearly, $\hat{G}_{i,j}$ is a good 2-connected outer-1-plane graph if $G$ is such a graph.

A vertex set $\mathcal{V}[v_i,v_j]$ with $i\neq j$ is a \emph{non-edge} if $j=i+1$ and $v_iv_j\not\in E(G)$, and is a \emph{path} if $v_k v_{k+1}\in E(G)$ for all $i\leq k<j$. An edge $v_iv_j$ in $G$ is a \emph{chord} if $|j-i|\neq 1$ or $|G|-1$. By $\mathcal{C}[v_i,v_j]$, we denote the set of chords $xy$ with $x,y\in \mathcal{V}[v_i,v_j]$.

In any figure of this paper, the degree of a solid (or hollow) vertex is exactly (or at least) the number of edges that are incident with it, respectively, and a solid vertex is distinct to every another vertex
but two hollow vertices may be identified unless stated otherwise.

We now collect some useful results that will be applied in the next sections.

\begin{lem}\label{path}{\rm \cite[Claim 1]{ZPOPG}}
Let $v_a$ and $v_b$ be vertices of $G$. If there are no crossed chords in $\mathcal{C}[v_a,v_b]$ and no edges between
$\mathcal{V}(v_a,v_b)$ and $\mathcal{V}(v_{b},v_{a})$, then $\mathcal{V}[v_a,v_b]$ is either non-edge or path.
\end{lem}

In what follows, when mentioning the configuration $G_i$ with $1\leq i\leq 14$ or $S_i$ with $1\leq i\leq 3$ we always refer to the corresponding picture in Figures \ref{str} or \ref{s1s2s3}.

Saying that $G$ \emph{contains} $G_i$ or $S_i$, we mean that $G$ contains a subgraph isomorphic to $G_i$ or $S_i$ such that the degree in $G$ of any solid (resp.\,hollow) vertex in that picture is exactly (resp.\,at least) the number of edges that are incident with it there.

For two distinct vertices $v_a$ and $v_b$ on the outer boundary of $G$, saying $G[v_a,v_b]$ \emph{properly contains} $G_i$ or $S_i$, we mean that $G[v_a,v_b]$ contains $G_i$ or $S_i$ so that neither $v_a$ nor $v_b$ corresponds to a solid vertex or a hollow vertex with a degree restriction in the picture of $G_i$ or $S_i$.

\begin{lem}\label{noncrossed}
Let $\mathcal{V}[v_a,v_b]$ with $b-a\geq 3$ be a path in $G$. If $\Delta(G)\leq 4$ and there are no crossed chords in $\mathcal{C}[v_a,v_b]$ and
no edges between $\mathcal{V}(v_a,v_b)$ and $\mathcal{V}(v_{b},v_{a})$, then $G[v_a,v_b]$ properly contains $G_1$ or $G_2$.
\end{lem}

\begin{proof}
If $\mathcal{C}[v_a,v_b]\setminus \{v_av_b\}=\emptyset$ (note that the chord $v_av_b$ may not really exist), then $d(v_{a+1})=d(v_{a+2})=2$ and $G_1$ is properly contained. If
there is at least one chord in $\mathcal{C}[v_a,v_b]\setminus \{v_av_b\}$, then choose one, say $v_{a'}v_{b'}$ with $a\leq a'<b'\leq b$, so that there is no other chord in $\mathcal{C}[v_{a'},v_{b'}]$.
If $b'-a'\geq 3$, then $d(v_{a'+1})=d(v_{a'+2})=2$ and $G_1$ is properly contained.
If $b'-a'=2$, then $d(v_{a'+1})=2$. Choose $t\in \{a',b'\}$ such that $v_t\neq v_a,v_b$.
If $d(v_{t})\leq 3$, then $G_1$ is properly contained.
If $d(v_{t})=4$, then there is another one chord $v_{t}v_{c'}$ with $a\leq c'\leq b$ and $c'\neq a', b'$. If $|c'-t|=2$, then $d(v_{t-1})=d(v_{t+1})=2$, and thus $G_2$ is properly contained.
If $|c'-t|\geq 3$, then let $a:=\min\{c',t\}$, $b:=\max\{c',t\}$ and come back to the first line of this proof.
Since $t\neq a,b$, $|c'-t|<|b-a|$, which implies that this iterative process will terminate.
\end{proof}

\begin{lem}\label{pc}
Let $v_iv_j$ cross $v_kv_l$ in $G$ with $i<k<j<l$ so that there are no other crossed chords besides $v_iv_j$ and $v_kv_l$ in $\mathcal{C}[v_i,v_l]$. If $\Delta(G)\leq 4$ and $\max\{|\mathcal{V}[v_i,v_k]|,|\mathcal{V}[v_k,v_j]|,|\mathcal{V}[v_j,v_l]|\}\geq 4$, then $G[v_i,v_l]$ properly contains $G_1$ or $G_2$.
\end{lem}

\begin{proof}
Without loss of generality, assume that $\mathcal{V}[v_i,v_k]\geq 4$. This implies that $k-i\geq 3$.
Note that there are no edges between $\mathcal{V}(v_i,v_k)$ and $\mathcal{V}(v_k,v_i)$, since $G$ is outer-1-planar.
By Lemma \ref{path}, $\mathcal{V}[v_i,v_k]$ is a path. By Lemma \ref{noncrossed}, $G[v_i,v_k]$ properly contains $G_1$ or $G_2$.
Since for any vertex in $\mathcal{V}(v_i,v_k)$, its degree in $G[v_i,v_k]$ is the same as that in $G[v_i,v_l]$. Hence $G[v_i,v_l]$ properly contained $G_1$ or $G_2$.
\end{proof}

\begin{figure}
  \begin{center}
  \includegraphics[height=10.5cm,width=15cm]{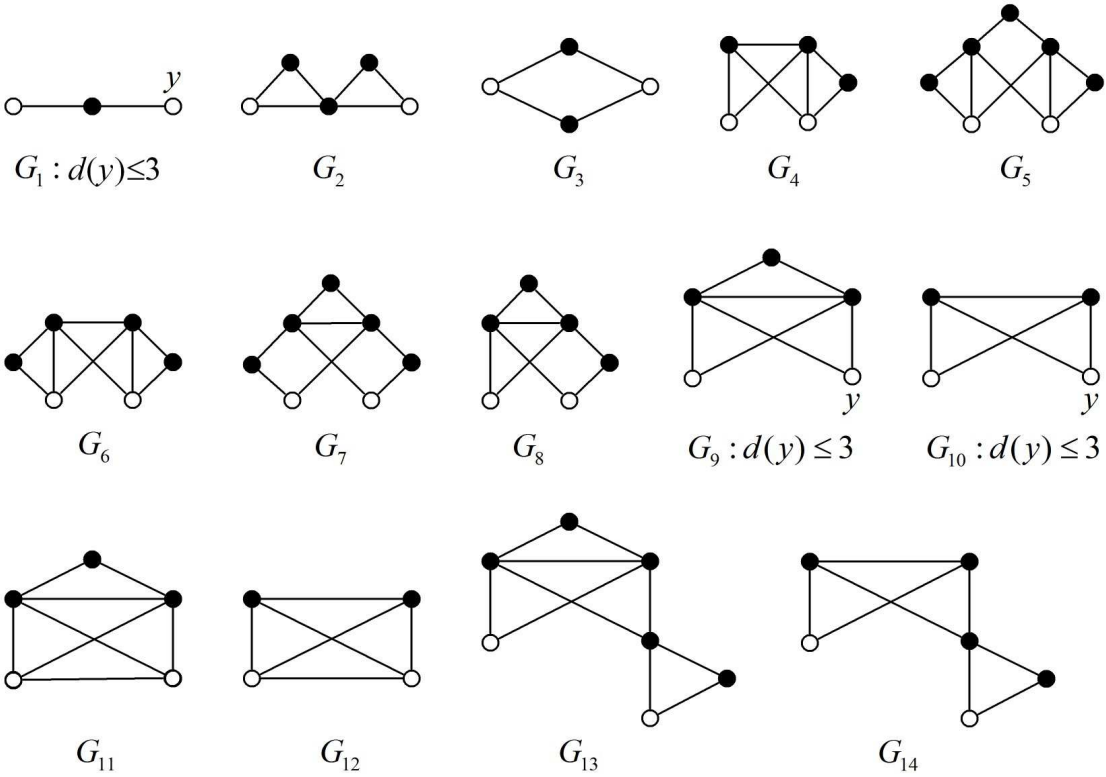}\\
  \end{center}
  \caption{Local structures in outer-1-planar graph with $\vartheta(G)\geq 1$ and $\Delta(G)\leq 4$}\label{str}
\end{figure}

\section{Local structures}

\begin{figure}
  \begin{center}
  \includegraphics[width=12cm,height=5cm]{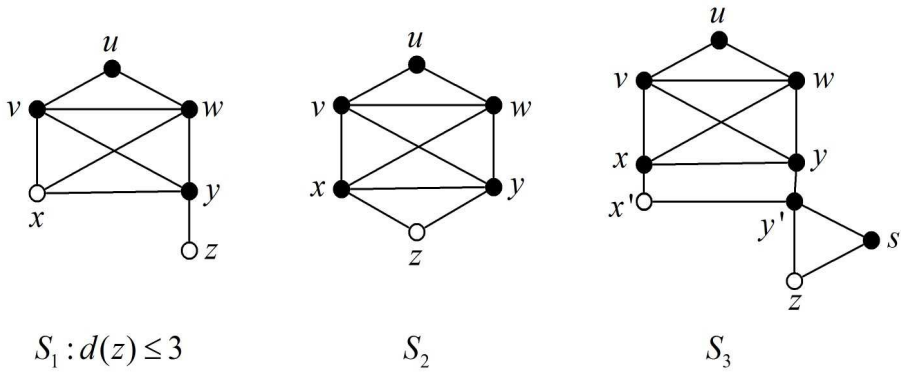}\\
  \end{center}
  \caption{Special configurations}\label{s1s2s3}
\end{figure}

\begin{lem}\label{mainlemma}
Let $G$ be a good $2$-connected outer-1-plane graph with vertices $v_1,v_2,\ldots,v_n$ lying clockwise on its outer boundary, where $n=|G|$. If $\Delta(G)\leq 4$ and $\vartheta(G)\geq 1$, then $G[v_1,v_n]$ properly contains one of the configurations (see Figures \ref{str} and \ref{s1s2s3})\\
\indent (1) $G_1$, $G_3$ if $n=4$, unless $\mathcal{V}[v_1,v_4]$ is a path and $v_1v_3,v_2v_4\in E(G)$; \\
\indent (2) $G_1,\ldots,G_4,G_{10}$, $G_{12}$ if $n=5$, unless $\mathcal{V}[v_1,v_5]$ is a path and $v_1v_4,v_2v_4,v_2v_5\in E(G)$;\\
\indent (3) $G_1,\ldots,G_4,G_6,G_8,\ldots,G_{12},G_{14}$ if $n=6$;\\
\indent (4) $G_1,\ldots,G_{14}$  if $n\geq 7$;\\
\indent (5) $G_1,\ldots,G_{10},G_{12},G_{13},G_{14}$, $S_1$, $S_3$ if $n\geq 8$ and $\vartheta(G)\geq 3$.
\end{lem}

\begin{proof}
If there is no crossing in $G$, then $v_1v_2\ldots v_n$ forms a path since $G$ is 2-connected. Under this condition, one can easily show, by Lemma \ref{noncrossed}, that $G[v_1,v_n]$ properly contains $G_1$ if $n=4$, and $G_1$ or $G_2$ if $n\geq 5$. Hence in the following we always assume that there is at least one crossing in $G$.

\textbf{Case 1. $n=4$.}

Suppose that $v_1v_3$ crosses $v_2v_4$. Since $G$ is good and 2-connected, at least two of $v_1v_2,v_2v_3$ and $v_3v_4$ belong to $E(G)$.
If $v_2v_3\not\in E(G)$, then $v_1v_2,v_3v_4\in E(G)$ and $G_3$ is properly contained. If $v_2v_3\in E(G)$ and $\{v_1v_2,v_3v_4\}\not\subseteq E(G)$, then $G_1$ is properly contained
If $v_1v_2,v_2v_3,v_3v_4\in E(G)$, then $\mathcal{V}[v_1,v_4]$ is a path.

\textbf{Case 2. $n=5$.}

By symmetry, we consider three cases. First, if $v_1v_3$ crosses $v_2v_4$, then at least two of $v_1v_2,v_2v_3$ and $v_3v_4$ belong to $E(G)$, since $G$ is good and 2-connected.
If $v_2v_3\not\in E(G)$, then $v_1v_2,v_3v_4\in E(G)$ and $G_3$ is properly contained. If $v_2v_3\in E(G)$ and $\{v_1v_2,v_3v_4\}\not\subseteq E(G)$, then $G_1$ is properly contained
If $v_1v_2,v_2v_3,v_3v_4\in E(G)$, then $G_{12}$ is properly contained if $v_1v_4\in E(G)$, and $G_{10}$ is contained if $v_1v_4\not\in E(G)$.

Second, if $v_1v_3$ crosses $v_2v_5$, then $v_3v_4,v_4v_5\in E(G)$ and $d(v_4)=2$, since $G$ is 2-connected. If $\{v_2v_3,v_3v_5\}\not\subseteq E(G)$, then $d(v_3)\leq 3$ and $G_1$ is properly contained. If $\{v_2v_3,v_3v_5\}\subseteq E(G)$, then $G[v_1,v_n]$ properly contains $G_3$ if $v_1v_2\not\in E(G)$, or $G_4$ if $v_1v_2\in E(G)$.

Third, if $v_1v_4$ crosses $v_2v_5$, then $v_2v_3,v_3v_4\in E(G)$ and $d(v_3)=2$, since $G$ is 2-connected. If $\{v_1v_2,v_2v_4,v_4v_5\}\not\subseteq E(G)$, then $\min\{d(v_2),d(v_4)\}\leq 3$ and $G_1$ is properly contained. If $\{v_1v_2,v_2v_4,v_4v_5\}\subseteq E(G)$, then the excluded case occurs.

\textbf{Case 3. $n=6$.}

Suppose that $v_iv_j$ crosses $v_kv_l$ with $1\leq i<k<j<l \leq 6$. In what follows, we consider three major cases. Note that there is no edge between $\mathcal{V}(v_i,v_l)$ and $\mathcal{V}(v_l,v_i)$, since $G$ is an outer-1-plane graph. So, the graph $\hat{G}_{i,l}$ is a good 2-connected outer-1-plane graph, and thus the results (1) and (2) can be applied to $\hat{G}_{i,l}$.

\textbf{Subcase 3.1. $|\mathcal{V}[v_i,v_l]|=4$.}

By (1), $\mathcal{V}[v_i,v_l]$ is a path, because otherwise $\hat{G}_{i,l}[v_i,v_l]$ properly contains $G_1$ or $G_3$
, and so does $G[v_1,v_6]$.

If $v_iv_l\in E(G)$, then $G_{12}$ is properly contained. Therefore, we consider the case that $v_iv_l\not\in E(G)$. By symmetry, we discuss the following two subcases.

If $i=1$, then $v_4v_5,v_4v_6\in E(G)$, because otherwise $d(v_4)\leq 3$ and $G_{10}$ is properly contained. This implies that $v_5v_6\in E(G)$ since $G$ is 2-connected. Therefore, $G_{14}$ is properly contained.

If $i=2$, then $d(v_2)\leq 3$ or $d(v_5)\leq 3$, because otherwise $v_2v_6$ crosses $v_1v_5$, contradicting the fact that $\vartheta(G)\geq 1$. Therefore, $G_{10}$ is properly contained.

\textbf{Subcase 3.2. $|\mathcal{V}[v_i,v_l]|=5$.}

Assume, without loss of generality, that $i=1$. By (2), $\mathcal{V}[v_1,v_5]$ is a path and $v_1v_4,v_2v_4,v_2v_5\in E(G)$, because otherwise $\hat{G}_{1,5}[v_1,v_5]$ (and thus $G[v_1,v_5]$) properly contains one configurations from the list in (2). If $v_1v_5\in E(G)$, then $G_{11}$ is properly contained. If $v_1v_5\not\in E(G)$, then $d(v_5)\leq 3$ and $G_9$ is properly contained.

\textbf{Subcase 3.3. $|\mathcal{V}[v_i,v_l]|=6$.}

If $\max\{|\mathcal{V}[v_i,v_k]|,|\mathcal{V}[v_k,v_j]|,|\mathcal{V}[v_j,v_l]|\}\geq 4$, then by the fact that $\vartheta(G)\geq 1$ and by Lemma \ref{pc}, $G[v_i,v_l]$ properly contains $G_1$ or $G_2$, and so does $G[v_1,v_n]$.

Hence we assume that $\max\{|\mathcal{V}[v_i,v_k]|,|\mathcal{V}[v_k,v_j]|,|\mathcal{V}[v_j,v_l]|\}\leq 3$. By symmetry we consider two subcases.

First, if $v_1v_4$ crosses $v_2v_6$, then by the 2-connectedness of $G$, we have $v_2v_3,v_3v_4,v_4v_5,v_5v_6\in E(G)$. This implies that $d(v_3)=2$. If $\{v_1v_2,v_2v_4\}\not\subseteq E(G)$, then $d(v_2)\leq 3$ and $G_1$ is properly contained. If $\{v_1v_2,v_2v_4\}\subseteq E(G)$, then $v_4v_6\not\in E(G)$ since $\Delta(G)\leq 4$, which implies that $G_8$ is properly contained.

Second, if $v_1v_4$ crosses $v_3v_6$, then by the the 2-connectedness of $G$, we have $v_1v_2,v_2v_3,v_4v_5,v_5v_6\in E(G)$ and $d(v_2)=d(v_5)=2$. If $\{v_1v_3,v_3v_4,v_4v_6\}\not\subseteq E(G)$, then $v_3$ or $v_4$ has degree at most 3 and $G_1$ is properly contained. If $\{v_1v_3,v_3v_4,v_4v_6\}\subseteq E(G)$, then $G_6$ is properly contained.

\textbf{Case 4. $n\geq 7$.}

We prove (4) by induction on $n$. First, we prove it for $n=7$ in Case 4.1, and then assume that the result holds for good $2$-connected outer-1-plane graphs $G$ with order $n'$, where $7\leq n'<n$. In Case 4.2, we prove (4) for $n\geq 8$, where the above induction hypothesis will be frequently applied.

\textbf{Case 4.1.} $n=7$.

Suppose that $v_iv_j$ crosses $v_kv_l$ with $1\leq i<k<j<l \leq 7$. Three major cases are considered as follows. Again, note that there is no edge between $\mathcal{V}(v_i,v_l)$ and $\mathcal{V}(v_l,v_i)$, since $G$ is an outer-1-plane graph. So, the graph $\hat{G}_{i,l}$ is a good 2-connected outer-1-plane graph, and thus the results (1), (2) and (3) can be applied to $\hat{G}_{i,l}$.

\textbf{Subcase 4.1.1. $|\mathcal{V}[v_i,v_l]|=4$.}

By (1), we assume that $\mathcal{V}[v_i,v_l]$ is a path and $v_iv_j,v_kv_l\in E(G)$ (the reason for this is the same as the one we stated in Subcase 3.1. Here and below, we do not repeatedly explain why and how the previous results be applied). If $v_iv_l\in E(G)$, then $G_{12}$ is properly contained. Therefore, we consider the case that $v_iv_l\not\in E(G)$.
By symmetry, the following two subcases are considered.

If $i=1$, then $d(v_4)=4$ (otherwise $G_{10}$ is properly contained), which implies that $v_4v_6\in E(G)$ or $v_4v_7\in E(G)$. If $v_4v_6\in E(G)$, then by the 2-connectedness of $G$, $v_4v_5,v_5v_6\in E(G)$ and $d(v_5)=2$, which implies that $G_{14}$ is properly contained. If $v_4v_6\not\in E(G)$ and $v_4v_7\in E(G)$, then $v_4v_5,v_5v_6,v_6v_7\in E(G)$ since $G$ is 2-connected, which implies that $d(v_6)=2$ and $d(v_5)\leq 3$. Hence $G_1$ is properly contained in $G[v_1,v_7]$.

If $i=2$, then $d(v_2)=d(v_5)=4$, because otherwise $G_{10}$ is properly contained. This case appears only if $v_1v_2,v_2v_7,v_5v_6,v_6v_7,v_5v_7\in E(G)$ since $\vartheta(G)\geq 1$, and thus $G_{14}$ is properly contained.

\textbf{Subcase 4.1.2. $|\mathcal{V}[v_i,v_l]|=5$.}

By symmetry, we consider two subcases.

First, if $i=1$, then by (2), we assume that $\mathcal{V}[v_1,v_5]$ is a path and $v_1v_4,v_2v_4,v_2v_5\in E(G)$. If $v_1v_5\in E(G)$, then $G_{11}$ is properly contained. Therefore we assume that $v_1v_5\not\in E(G)$. If $d(v_5)\leq 3$, then $G_9$ is properly contained.  If $d(v_5)=4$, then $v_5v_6,v_5v_7\in E(G)$, and by the 2-connectedness of $G$, we also have $v_6v_7\in E(G)$ and $d(v_6)=2$. This implies that $G_{13}$ is properly contained.

Second, if $i=2$, then by (2), we assume that $\mathcal{V}[v_2,v_6]$ is a path and $v_2v_5,v_3v_5,v_3v_6\in E(G)$. If $v_2v_6\in E(G)$, then $G_{11}$ is properly contained. If $v_2v_6\not\in E(G)$, then $d(v_2)\leq 3$ or $d(v_7)\leq 3$, because otherwise $v_1v_6$ crosses $v_2v_7$, contradicting the fact that $\vartheta(G)\geq 1$. Therefore, $G_9$ is properly contained.

\textbf{Subcase 4.1.3. $|\mathcal{V}[v_i,v_l]|=6$.}

By (3), $G[v_i,v_l]$ properly contains $G_1,\ldots,G_4,G_6,G_8,\ldots,G_{12},G_{14}$, and so does $G[v_1,v_7]$.

\textbf{Subcase 4.1.4. $|\mathcal{V}[v_i,v_l]|=7$.}

If $\max\{|\mathcal{V}[v_i,v_k]|,|\mathcal{V}[v_k,v_j]|,|\mathcal{V}[v_j,v_l]|\}\geq 4$, then by Lemma \ref{pc}, $G[v_i,v_l]$ properly contains $G_1$ or $G_2$, and so does $G[v_1,v_n]$. Here, note that there is no other crossed chords besides $v_iv_j$ and $v_kv_l$ in $\mathcal{C}[v_i,v_l]$, since $\vartheta(G)\geq 1$.

Hence we leave an unique case, that is, the case when $v_1v_5$ crosses $v_3v_7$. Since $G$ is 2-connected, $\mathcal{V}[v_1,v_7]$ is a path and $d(v_2)=d(v_4)=d(v_6)=2$. If $\{v_1v_3,v_5v_7\}\not\subseteq E(G)$, then $G[v_1,v_7]$ properly contains $G_1$ if $v_3v_5\not \in E(G)$, and $G_7$ otherwise. If $\{v_1v_3,v_5v_7\}\subseteq E(G)$, then $v_3v_5\not\in E(G)$ since $\Delta(G)\leq 4$, and thus $G_5$ is properly contained.

\textbf{Case 4.2.} $n\geq 8$.

Choose two mutually crossed chords $v_iv_j$ and $v_kv_l$ with $1\leq i<k<j<l \leq n$ so that $l-i$ is as minimum as possible. Clearly, there is no other crossed chord besides $v_iv_j$ and $v_kv_l$ in $\mathcal{C}[v_i,v_l]$ by this choice.

If $\max\{|\mathcal{V}[v_i,v_k]|,|\mathcal{V}[v_k,v_j]|,|\mathcal{V}[v_j,v_l]|\}\geq 4$, then by Lemma \ref{pc}, $G[v_i,v_l]$ properly contains $G_1$ or $G_2$, and so does $G[v_1,v_n]$. Hence we assume that $\max\{|\mathcal{V}[v_i,v_k]|,|\mathcal{V}[v_k,v_j]|,|\mathcal{V}[v_j,v_l]|\}\leq 3$ (i.e., $\max\{k-i,j-k,l-j\}\leq 2$). This implies that
$|\mathcal{V}[v_i,v_l]|\leq 7$.
If $|\mathcal{V}[v_i,v_l]|=7$ (resp.\,$|\mathcal{V}[v_i,v_l]|=6$),
then by Case 4.1 (resp.\,by (3)), $\hat{G}_{i,l}[v_i,v_l]$ properly contains one of the required configurations, and so does $G[v_1,v_n]$.

If $|\mathcal{V}[v_i,v_l]|=5$ (resp.\,$|\mathcal{V}[v_i,v_l]|=4$), then by (2) (resp.\,by (1)), we only consider the case that $\mathcal{V}[v_i,v_l]$ is a path with $j-k=2$ and $v_kv_j\in E(G)$ (resp.\,with $j-k=1$). If $v_iv_l\in E(G)$, then $G_{11}$ (resp.\,$G_{12}$) is properly contained. Therefore, we assume that
$v_iv_l\not\in E(G)$.

Since $i\neq 1$ or $l\neq n$, we assume, by symmetry, that $l\neq n$. If $d(v_l)\leq 3$, then $G_9$ (resp.\,$G_{10}$) is properly contained. If $d(v_l)=4$, then by the fact that $\vartheta(G)\geq 1$, there is a non-crossed chord $v_lv_s$ with $1\leq s\leq n$ and $s\neq i,k$.

\textbf{Subcase 4.2.1. $i\neq 1$.}

It is easy to see that $l<s\leq n$ or $1\leq s<i$. If $1\leq s<i$, then $d(v_i)=4$, because otherwise $G_9$ (resp.\,$G_{10}$) is properly contained. This implies that there is a non-crossed chord $v_iv_t$ with $s\leq t<i$, since $\vartheta(G)\geq 1$.
Therefore, we shall then consider two major cases: (a) there is a non-crossed chord $v_lv_s$ with $l<s\leq n$, or (b) there is non-crossed chord $v_iv_t$ with $s\leq t<i$. Clearly, this two cases are symmetry. Hence we just need consider one, say (a).

Choose one $s$ from those satisfying the condition (a) so that $s-l$ is as large as possible.
At this moment, there is no chord in the form $v_lv_t$ with $1\leq t<i$, because otherwise $v_s$ would be a cut vertex separating $\mathcal{V}(v_t,v_s)$ from $\mathcal{V}(v_s,v_t)$. Hence by the choice of $s$ and the fact that $v_lv_s$ is non-crossed, there is no edge between $\mathcal{V}(v_i,v_s)$ and $\mathcal{V}(v_s,v_i)$.

If $s-l=2$, then by the 2-connectedness of $G$, $v_lv_{l+1},v_{l+1}v_s\in E(G)$ and $d(v_{l+1})=2$, which implies that $G_{13}$ (resp.\,$G_{14}$) is properly contained. If $s-l\geq 3$, then $\hat{G}_{i,s}$ is a good 2-connected outer-1-plane graph with order $n'$, where $7\leq n'=s-i+1<s\leq n$ (note that $l-i\geq 3$). By the induction hypothesis, $\hat{G}_{i,s}[v_i,v_s]$ properly contains one of the configurations among $G_1,\ldots,G_{14}$. Since there is no edge between $\mathcal{V}(v_i,v_s)$ and $\mathcal{V}(v_s,v_i)$, any configuration properly contained in $\hat{G}_{i,s}[v_i,v_s]$ is properly contained in $G[v_i,v_s]$, and then in $G[v_1,v_n]$.

\textbf{Subcase 4.2.2. $i=1$.}

In this case we have $l<s\leq n$. Choose such an $s$ so that  $s-l$ is as large as possible. If  $s-l=2$, then by the 2-connectedness of $G$, $v_lv_{l+1},v_{l+1}v_s\in E(G)$ and $d(v_{l+1})=2$, which implies that $G_{13}$ (resp.\,$G_{14}$) is properly contained. If $s-l\geq 3$ and $s\neq n$, then applying the induction hypothesis to the graph $\hat{G}_{i,s}$ and we can obtain the required result as we have done in Subcase 4.2.1. Note that there is no edge between $\mathcal{V}(v_i,v_s)$ and $\mathcal{V}(v_s,v_i)$ by the choice of $s$.
At last, we are left the case that $s-l\geq 3$ and $s=n$.

If there is a pair of crossed chords $v_{i'}v_{j'}$ and $v_{k'}v_{l'}$ with $l\leq i'<k'<j'<l'\leq s$, then before augmenting, we can properly choose such a pair so that $l'-i'$ is as minimum as possible. Now, we come back to the first line of Case 4.2 by setting $i:=i',j:=j',k:=k'$ and $l:=l'$. Note that $i'\neq 1$, and thus the new $i$ is not 1. Therefore, this second round of arguments will not involve the current subcase and thus one of the required configurations can be properly contained in $G[v_1,v_n]$.

Hence we assume that there are no crossed chords in $\mathcal{C}[v_l,v_s]$. Since $v_lv_s$ is not crossed, there are no edges between $\mathcal{V}(v_l,v_s)$ and $\mathcal{V}(v_s,v_l)$. Since $s-l\geq 3$, $\mathcal{V}[v_i,v_l]$ is a path by Lemma \ref{path}. So, by Lemma \ref{noncrossed}, $G[v_l,v_s]$ properly contains $G_1$ or $G_2$, and so does $G[v_1,v_n]$.

\textbf{Case 5. $n\geq 8$ and $\vartheta(G)\geq 3$.}


We prove (5) for by induction on $n$. First, we prove it for $n=8$ in Case 5.1, and then assume that the result holds for good $2$-connected outer-1-plane graphs $G$ with order $n'$, where $8\leq n'<n$. In Case 5.2, we prove (5) for $n\geq 9$, where the above induction hypothesis will be frequently applied.

\textbf{Case 5.1.} $n=8$.

We assume that $G[v_1,v_n]$ properly contains $G_{11}$, as otherwise the result follows from (4). Let $G[v_i,v_l]\cong G_{11}$, where $i\leq i<l\leq n$. Since $G$ is 2-connected, $v_i,v_l$ both have degree 4 (otherwise one of them is a cut-vertex of $G$), and there exist chords $v_lv_s$ and $v_tv_i$ such that $s\neq i$ and $t\neq l$. If $s<i$, then $v_s$ becomes a cut-vertex of $G$ unless $s=1$ and $l=n$, since $v_lv_s$ is non-crossed by the face that $\vartheta(G)\geq 1$. Hence $l<s\leq n$ if $l\neq n$. Similar, $1\leq t<i$ if $i\neq 1$. Since $l-i=4$, either $i\neq 1$ or $l\neq n$. Without loss of generality, we assume the latter, and thus $l<s\leq n$. If $s-l\geq 2$, then $v_s$ is a cut-vertex separating $v_l$ from $v_{l+1}$, contradicting the 2-connectedness of $G$. Hence $s=l+1$.

If $v_iv_{s}\in E(G)$, then $v_{s}$ is a cut-vertex separating $\mathcal{V}[v_i,v_l]$ from $V(G)\setminus \mathcal{V}[v_i,v_{s}]$, contradicting the 2-connectedness of $G$. Hence $v_iv_{s}\not\in E(G)$ and $d(v_{s})\leq 3$, since $n=8$. This implies that $G[v_1,v_8]$ properly contains $S_1$.

\textbf{Case 5.2.} $n\geq 9$.

As in Case 5.1, we may assume that $G[v_i,v_l]\cong G_{11}$, where $l\leq i<l\leq n$. Since $l-i=4$, either $i\neq 1$ or $l\neq n$. Without loss of generality, assume that $l\neq n$. By the same argument as in Case 5.1, we can prove that $v_lv_{l+1}\in E(G)$, $v_{i}v_{i-1}\in E(G)$ if $i\geq 2$, and $v_1v_n\in E(G)$ if $i=1$.

If we meet the case that $i=1$ and $v_1v_n\in E(G)$, then $l+1=6<n$. By relabelling the vertices of $G$ from $v_1,v_2,\ldots,v_{n-1},v_n$ to $v_2,v_3,\ldots,v_n,v_1$, we translate this case to the one that $i\geq 2$, $l\neq n$ and $v_lv_{l+1},v_{i}v_{i-1}\in E(G)$.

Therefore, we assume $i\geq 2$ and $v_iv_{i-1}\in E(G)$ in the following.

Since $n\geq 9$ and $|\mathcal{V}[v_{i-1},v_{l+1}]|=7$, either $i-1\neq 1$ or $l+1\neq n$. Without loss of generality, assume the latter. If $d(v_{l+1})\leq 3$, then $S_1$ is properly contained. If $d(v_{l+1})=4$, then there is a chord $v_{l+1}v_s$ with $1\leq s\leq n$ so that $s>l+1$ or $s<i-1$. Since $\vartheta(G)\geq 3$, $v_{l+1}v_s$ is non-crossed.

If $s<i-1$, then $i-1\neq 1$ and thus $d(v_{i-1})=4$ as otherwise $S_1$ is properly contained. This implies that there is a chord $v_{i-1}v_t$ with $s\leq t<i-1$.
Therefore, there is either (a) a chord $v_{l+1}v_s$ with $l+1<s\leq n$, or (b) a chord $v_{i-1}v_t$ with $1\leq t<i-1$. By symmetry, we assume that (a) exists. Among those $s$ satisfying (a), choose one $s$ so that $s-(l+1)$ is as large as possible.

If there is a chord $v_tv_{l+1}$ so that $1\leq t<i-1$, then $v_{t}v_{l+1}$ is non-crossed since $\vartheta(G)\geq 3$, which implies that there is no edge between $\mathcal{V}(v_t,v_{l+1})$ and $\mathcal{V}(v_{l+1},v_t)$.
Therefore, $\hat{G}_{t,l+1}$ is a good 2-connected outer-1-planar graph with order $n'$, where $8\leq n'=(l+1)-t+1<n$ (note that $t<i-1$ and $l-i=4$), and then by the induction hypothesis, $\hat{G}_{t,l+1}[v_t,v_{l+1}]$ properly contains one of the configurations from the list in (5). Since no edge exists between $\mathcal{V}(v_t,v_{l+1})$ and $\mathcal{V}(v_{l+1},v_t)$, any configuration properly contained in $\hat{G}_{t,l+1}[v_t,v_{l+1}]$ is properly contained in $G[v_t,v_{l+1}]$, and then in $G[v_1,v_n]$.

Hence we assume that there is no chord $v_tv_{l+1}$ with $1\leq t<i-1$. At this stage, there is no edge between $\mathcal{V}(v_{i-1},v_s)$ and $\mathcal{V}(v_{s},v_{i-1})$ by the choice of $s$ and the fact that $v_{l+1}v_s$ in non-crossed.

If $s=l+3$, then $v_{l+1}v_{l+2},v_{l+2}v_s\in E(G)$ and $d(v_{l+2})=2$ by the 2-connectedness of $G$. Since $v_{l+1}$ has degree 4, by the choice of $s$, we shall have $v_{i-1}v_{l+1}\in E(G)$, which implies that $S_3$ is properly contained in $G[v_{i-1},v_s]$, and thus in $G[v_{1},v_n]$.

If $s=l+4$, then by (1), $G[v_{l+1},v_s]$ (and so does $G[v_1,v_n]$) properly contains $G_1$ or $G_3$, unless $v_{l+1}v_{l+3}$ crosses $v_{l+2}v_s$, which case contradicts the fact that $\vartheta(G)\geq 3$.


If $s=l+5$, then by (2), $G[v_{l+1},v_s]$ (and so does $G[v_1,v_n]$) properly contains one of the configurations among $G_1,\ldots,G_4,G_{10}$,$G_{12}$, unless $v_{l+1}v_{l+4}$ crosses $v_{l+2}v_5$, which case contradicts the fact that $\vartheta(G)\geq 3$.


If $s=l+6$, then by (3), $G[v_{l+1},v_s]$ (and so does $G[v_1,v_n]$) properly contains one of the configurations among $G_1,\ldots,G_4,G_6,G_8,G_9,G_{10},G_{12},G_{14}$, unless $G[v_{l+2},v_s]\cong G_{11}$, which case contradicts the fact that $\vartheta(G)\geq 3$..

If $s=l+7$, then by (4), $G[v_{l+1},v_s]$ (and so does $G[v_1,v_n]$) properly contains one of the configurations among $G_1,\cdots,G_{10},G_{12},G_{13},G_{14}$, unless $G[v_{l+2},v_{l+6}]\cong G_{11}$ and $v_{l+1}v_{l+2},v_{l+6}v_s\in E(G)$. This contradicts the fact that $\vartheta(G)\geq 3$.

If $s\geq l+8$, then applying the induction hypothesis to the graph $G[v_{l+1},v_s]$, which has order $8\leq s-l<n$, we then conclude that $G[v_{l+1},v_s]$ properly contains at least one configuration from the list in (5), and so does $G[v_1,v_n]$.
\end{proof}

%
%
%
%
%

\begin{thm}\label{mainthm-2}
Every outer-1-planar graph $G$ with crossing distance at least 3 contains one of the configurations $G_1,\ldots,G_{10},G_{12},G_{13},G_{14}$, $S_1,S_2$ or $S_3$ as long as $\Delta(G)\leq 4$ and $\delta(G)\geq 2$.
\end{thm}

\begin{proof}
If $G$ is 2-connected, then let $H=G$, otherwise let $H$ be an end-block of $G$ (i.e., a 2-connected subgraph having only one cut-vertex of $G$).
Clearly $|H|\geq 3$ since $\delta(G)\geq 2$.
Let $u_1,u_2,\ldots,u_{|H|}$ be the vertices of $H$ with clockwise ordering on the boundary, where $u_1$ is the cut-vertex of $G$ if $H$ is an end-block of $G$.

If $|H|\geq 8$, then by Lemma \ref{mainlemma}(5), $H[u_1,u_{u_{|H|}}]$ properly contains one of those configurations, which is also contained in $G$.

If $6\leq |H|\leq 7$, then by Lemmas \ref{mainlemma}(3) and \ref{mainlemma}(4), $H[u_1,u_{u_{|H|}}]$ properly contains the configurations $G_1,\ldots,G_{10},G_{12},G_{13}$ or $G_{14}$ (and then $G$ contains one of them) if it does not properly contain $G_{11}$.
Suppose that $H[u_1,u_{u_{|H|}}]$ properly contains $G_{11}$. Let $H[u_i,u_l]\cong G_{11}$ and assume, without loss of generality, that $i<l$.
Suppose there is an edge $v_lv_s\in E(H)$ with $s\neq i,j,k$, because otherwise $v_l$ has degree 3 in $G$ and thus $G$ contains $G_9$. If $s\neq 1$, then $u_s$ has degree at most 3 in $G$ since $|H|\leq 7$, and thus $G$ contains $S_1$. If $s=1$ and $|H|=7$, then $i=3$ and $l=7$, because otherwise $u_1$ would be a cut-vertex of $H$, contradicts the 2-connectedness of $H$. This implies $u_1u_2,u_2u_3\in E(H)$ and thus $H[u_1,u_{u_{|H|}}]$ properly contains $S_1$, which is also contained in $G$. If $s=1$ and $|H|=6$, then by the 2-connectedness of $H$,
$i=2, l=6$ and $u_1u_2\in E(H)$, which implies that $H$ is isomorphic to the graph $S_2$ with $u_1$ corresponding to the vertex $z$ in that picture of $S_2$ in Figure \ref{s1s2s3}, and then $G$ contains $S_2$.

If $|H|=5$, then by Lemma \ref{mainlemma}(2), $H[u_1,u_5]$ properly contains $G_1,\ldots,G_4$ or $G_{12}$, and then $G$ contains one of them, or $\mathcal{V}[u_1,u_5]$ is a path and $u_1u_4,u_2u_4,u_2u_5\in E(G)$. In the latter case, we have $d(u_5)\leq 3$ and then $G$ contains $G_9$.

If $|H|=4$, then by Lemma \ref{mainlemma}(1), $H[u_1,u_4]$ properly contains $G_1$ or $G_3$, and then $G$ contains one of them, or $\mathcal{V}[u_1,u_4]$ is a path and $u_1u_3,u_2u_4\in E(G)$.
In this case, we have $d(u_4)\leq 3$ and then $G$ contains $G_{10}$.

If $|H|=3$, then $u_2$ and $u_3$ are two adjacent vertices of degree 2 in $G$. Hence $G$ contains $G_1$.
\end{proof}

\section{List coloring results}

\begin{figure}
  \begin{center}
  \includegraphics[width=15cm]{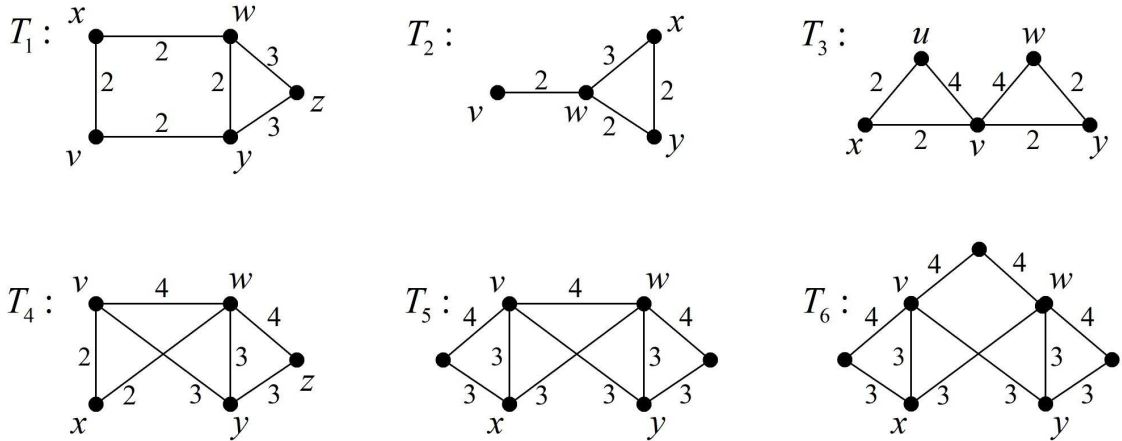}\\
  \end{center}
  \caption{Six list-colorable graphs}\label{list-colorable-graphs}
\end{figure}

Each of the graphs $T_i$ in Figure 1 is assumed to have an associated function $L_i$ that assigns a list of colors to each edge, of the size indicated by that edge.

\begin{thm}\label{six-useful-graphs}
(a) If $L_1(wx)\cap L_1(vy)=\emptyset$, then $T_1$ can be $L_1$-colored.\\
(b) For each integer $2\leq i\leq 6$, the graph $T_i$ is $L_i$-colorable.
\end{thm}

\begin{proof}

We consider each graph separately.

\textbf{Case 1:} $T_1$. If the edges $wx,wy,vy$ are $L_1$-colored with colors $a, b, c$, say, and this coloring cannot be extended to $wz$ and $yz$, then there is a color $d$ such that $L_1(wz) = \{a, b, d\}$ and $L_1(yz) = \{b, c, d\}$. These lists then determine $a$ and $c$ uniquely, as $\{a\} = L_1(wz) \setminus L_1(yz)$ and $\{c\} = L_1(yz) \setminus L_1(wz)$.

There are four different ways in which the edges $wx$ and $vy$ can be $L_1$-colored. Since $L_1(wx)\cap L_1(vy) =\emptyset$, at most one of these four colorings uses both colors in $L_1(vx)$, at most one uses both colors in $L_1(wy)$, and (as we have just seen) at most one can be extended to $wy$ and $vx$ but cannot then be extended to $wz$ and $yz$. Thus at least one of the four $L_1$-colorings of $wx$ and $vy$ can be extended to an $L_1$-coloring of $T_1$.

\textbf{Case 2:} $T_2$. If we can give $vw$ and $xy$ the same color, or alternatively give $wy$ a color not in $L_2(vw)$, then the remaining edges are easily colored. So we may assume $L_2(vw)\cap L_2(xy) = \emptyset$
and $L_2(vw) = L_2(wy)$. But then $L_2(wy) \cap L_2(xy) = \emptyset$, and any $L_2$-coloring of the three edges at $w$ can be extended to $xy$.

\textbf{Case 3:} $T_3$. Let $P$ denote the path $uxvyw$. $L_3$-color the edges in order along $P$. Then there is at least one color available for each of $uv$ and $vw$, and the only problem is if there is exactly one and it is the same color, say $d$, in each case. If this happens, then the colors along $P$ take the form $a, b, c, a'$ and $L_3(uv) = \{a, b, c, d\}$ and $L_3(vw) = \{a', b, c, d\}$, where possibly $a' = a$ but otherwise the colors are distinct. We may assume that $L_3(wy) = \{a', c\}$, as otherwise we could change the color of $wy$ and color $uv$ and $vw$ with $d$ and $a'$; and we may assume that $L_3(vy) = \{b, c\}$, as otherwise we could change the color of $vy$ and color $uv$ and $wy$ with $c$ and $vw$ with an available color. In a similar way, we may assume that $L_3(ux) = \{a, b\}$ and $L_3(vx) = \{b, c\}$. But then we can recolor the edges of $P$ with $b, c, b, c$
and $uv, vw$ with $a, d$.

\textbf{Case 4:} $T_4$. If possible, $L_4$-color $vy$ and $wx$ with the same color, and then color $vx$; the remaining edges can then be colored since $T_2$ is $L_2$-colorable. So we may assume that
$L_4(vy) \cap L_4(wx) = \emptyset$.

If possible, $L_4$-color $vx$ and $wy$ with the same color, and then color $wx$; the remaining edges can then be colored since they form a 4-cycle with at least two colors available for each edge, and a 4-cycle is edge-2-choosable. So we may assume that $L_4(vx) \cap L_4(wy) =\emptyset$.

If possible, $L_4$-color $vw$ with a color not in $L_4(vx) \cup L_4(wx)$; the remaining edges can then be colored since $T_1$ is $L_1$-colorable. So we may assume that $L_4(vw) = L_4(vx) \cup L_4(wx)$,
say $L_4(vx) = \{a,b\}$, $L_4(wx) = \{c, d\}$, and $L_4(vw) = \{a, b, c, d\}$.

If possible, $L_4$-color $vy$ and $wy$ with distinct colors not in $L_4(vw)$; then $vx$ and $wx$ are easily colored (they can be treated as non-adjacent, as their lists are disjoint), and the remaining edges can be colored in the order $yz,wz,vw$. So we may assume that there is a color $e$ such that $L_4(vy) = \{a, b, e\}$ and $L_4(wy) = \{c, d, e\}$.

Now choose a color $p \in \{a, b, c, d, e\}\setminus L_4(wz)$. It is easy to see that the edges not incident with $z$ can be $L_4$-colored so that $p$ is used on an edge at $w$. (If $p\in \{a, b, c, d\}$ then we can
use $p$ on $vw$, and if we use $c$ or $d$ on $vw$ then $e$ is used on $wy$.) Then the remaining edges $yz,wz$ can be colored in this order.

\textbf{Case 5:} $T_5$. If possible, $L_5$-color $vy$ and $wx$ with the same color, leaving at least three possible colors for $vw$. Since $T_2$ is $L_2$-colorable, at most one of these colors for $vw$ cannot be
extended to the triangle on the right, and at most one cannot be extended to the triangle on the left, and so at least one can be extended to both triangles, giving an $L_5$-coloring of $T_5$.

So we may assume that $L_5(vy) \cap L_5(wx) = \emptyset$. There are then nine different ways in which the two edges $vy$ and $wx$ can be $L_5$-colored. At most four of these ways use two colors
from $L_5(vw)$, the worst case being when $L_5(vw)$ contains two colors from $L_5(vy)$ and two from $L_5(wx)$. Similarly, at most two of these ways use two colors from $L_5(vx)$ and at most two use two colors from $L_5(wy)$. So there is at least one way of $L_5$-coloring $vy$ and $wx$ with two colors that are not both in $L_5(vw)$, not both in $L_5(vx)$, and not both in $L_5(wy)$. The remaining edges can then be colored in the same way as before.

\textbf{Case 6:} $T_6$. As in Case 5, $L_6$-color $vy$ and $wx$ either with the same color, or with two different colors that are not both in $L_6(vx)$ and not both in $L_6(wy)$. The remaining edges
can be colored by applying Case 2 twice.
\end{proof}

\begin{figure}
  \begin{center}
  \includegraphics[width=16cm,height=17cm]{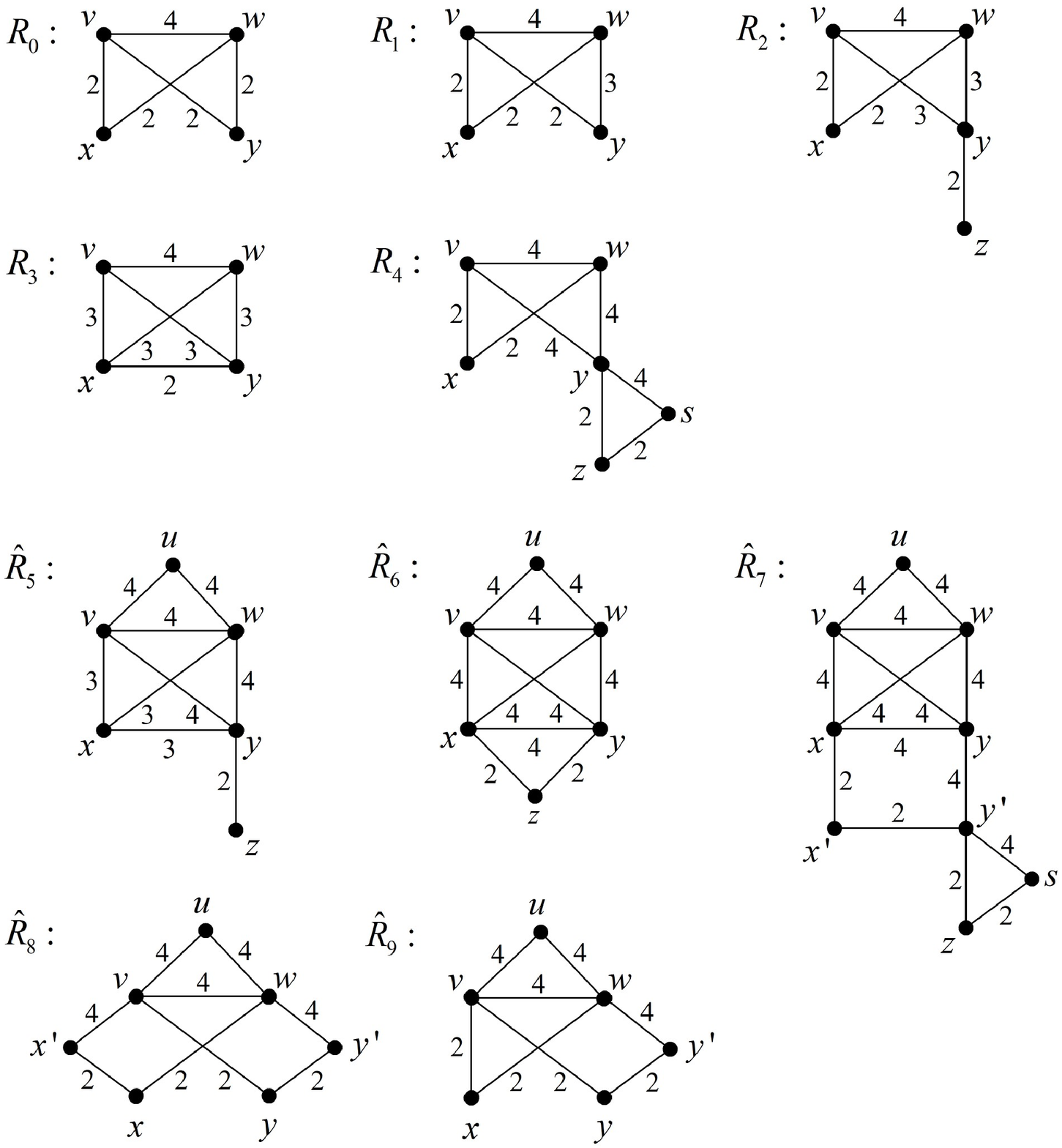}\\
  \end{center}
  \caption{More graphs for list-coloring}\label{more}
\end{figure}

Each of the graphs in Figure \ref{more} is assumed to have an associated function, which for
convenience we denote by the same letter $L$ in each case, that assigns a list of colors to each
edge, of the size indicated by that edge. In each graph, let $F$ denote the 4-cycle $vxwy$, with
edges $vx,wx,wy,vy$ in that order. Note that $F$ will be frequently used in the next proofs.

\begin{thm}\label{R-0-1-2-3}
(a) $R_0$ is $L$-colorable unless each color in $L(vw)$ is in the lists of exactly two adjacent edges of $F$;\\
(b) Each of the graphs $R_1,R_2,R_3$ is $L$-colorable.
\end{thm}

\begin{proof}
We consider each graph separately.

\textbf{Case 0:} $R_0$. Since a 4-cycle is edge-2-choosable, the edges of $F$ can be $L$-colored, with
colors $a, b, c, d$, say. Assume that $vw$ cannot now be colored. Then $a, b, c, d$ are all different
and $L(vw) = \{a, b, c, d\}$. If the list of some edge of $F$ has a color not in $\{a, b, c, d\}$ then we
can recolor that edge with such a color and use the freed color for $vw$; this contradiction
shows that all lists are subsets of $\{a, b, c, d\}$. If the lists of two opposite (i.e., non-adjacent)
edges of $F$ have a color in common, then we can color those two edges with that color and
the coloring is easily completed; this contradiction shows that each color occurs in the lists
of two adjacent edges of $F$, which proves (a).

\textbf{Case 1:} $R_1$. Delete a color from $L(wy)$ so that the list of every edge of $F$ now has two
colors. If these lists are of the form described in (a), delete a different color from $L(wy)$
instead; the coloring can now be completed.

\textbf{Case 2:} $R_2$. We will color $yz$ first; for $e\in \{vy,wy\}$, let $L_0(e)$ be obtained from $L(e)$ by
deleting the color of $yz$; for every other edge of $R_2-yz$, let $L_0(e) = L(e)$. We must color
$yz$ so that $R_2-yz$ has an $L_0$-coloring. We may assume that $L(yz)\subseteq L(vy) \cap L(wy)$, as
otherwise we can color $yz$ so that at least one of $L_0(vy)$ and $L_0(wy)$ contains three colors,
and the result will follow by Case 1 (possibly with $v$ and $w$ interchanged).

Let $L(yz) = \{p, q\}$. Color $yz$ with $p$. If the coloring cannot be completed, then the
lists $L_0$ have the form described in (a), with $L(vw) = L_0(vw) = \{a, b, c, d\}$ and each of
these colors being in the lists of two adjacent edges of $F$. Note that $q$ is in both $L_0(vy)$
and $L_0(wy)$, and $p$ is in neither of these lists, so that either (i) $p \in L_0(vx)\cap L_0(wx)$ or (ii)
$p \in  \{a, b, c, d\}$. If we recolor $yz$ with $q$ then we can complete the coloring, either by using $p$
on two nonadjacent edges of $F$, in case (i), or by using it on one edge, in case (ii).

\textbf{Case 3:} $R_3$. Let $L(vw) = \{a, b, c, d\}$. If possible, use the same color on $xy$ and $vw$; the
edges of the 4-cycle $vxwy$ can then be colored from the remaining lists of at least two colors.
So we may assume that $L(xy) = \{p, q\}$, where $p, q \not\in \{a, b, c, d\}$.

If the list of some edge of the 4-cycle $F$ contains $p$, first color $xy$ with $q$ and then extend
the coloring to $F$ so that some edge has color $p$; now we can color $vw$. If no edge of $F$
has $p$ in its list, color $xy$ with color $p$, and extend this coloring to the remaining edges by
Case 1.
\end{proof}

For each graph $R_i$ in Figure \ref{more}, let $\widehat{R}_i$ denote the graph obtained from $R_i$ by adding a
new vertex $u$ with neighbors $v$ and $w$, and let $L$ be extended to $\widehat{R}_i$ by giving lists of size 4
to the new edges $uv$ and $uw$.

\begin{thm}\label{R-1-2-bar}
(a) $\widehat{R}_1$ is $L$-colorable;\\
(b) $\widehat{R}_2$ is $L$-colorable if (i) some color is in the lists of two opposite (non-adjacent) edges of
$F$, or (ii) $L(vy)=L(wy)$.
\end{thm}

\begin{proof}
We know from Theorem \ref{R-0-1-2-3} that $R_1$ and $R_2$ are $L$-colorable; so consider an $L$-
coloring of one of them; we will try to extend it to $\widehat{R}_1$ or $\widehat{R}_2$ as appropriate. Let $p$ be the
color given to $yz$ in $R_2$. For $e \in \{vy,wy\}$ in $\widehat{R}_2$, let $L_0(e) = L(e) \setminus \{p\}$; for every other edge
of $\widehat{R}_2$, and every edge of $\widehat{R}_1$, let $L_0(e) = L(e)$. As in Theorem \ref{R-0-1-2-3}, we will assume in $\widehat{R}_2$
that $L(yz)\subseteq L(vy)\cap L(wy)$, as otherwise we can color $yz$ so that at least one of $L_0(vy)$ and
$L_0(wy)$ contains three colors, which is covered by the proof for $\widehat{R}_1$.

Let the edges $vx, xw,wy, yv$ of $F$ have colors $a, b, a', b'$ in this order, and let $vw$ have color
$c$, where possibly $a = a'$ and/or $b = b'$ but the colors are otherwise distinct. There is at
least one color available for each of $uv$ and $uw$, and the only problem arises if it is the same
color, say $d$, in each case, i.e., $L(uv) = \{a, b', c, d\}$ and $L(uw) = \{a', b, c, d\}$, where clearly
$d \not\in \{a, a', b, b¡ä, c\}$. Assume that this problem cannot be avoided by recoloring one or more
of the edges of $F$ and/or $vw$. Then the following hold.

\begin{description}
  \item[(O1)] $L_0(e) \subseteq \{a, a', b, b', c\}$ for each $e \in E(F)$, and $L(vw) \subseteq \{a, a', b, b', c, d\}$; otherwise
we can avoid the problem by recoloring a single edge.
  \item[(O2)] If $d \in L(vw)$ then $L_0(e) \subseteq \{a, a', b, b'\}$ for each $e \in E(F)$. To see this, recolor $vw$
with $d$ and apply (O1) with $c$ and $d$ interchanged.
  \item[(O3)] If $b \neq b'$  then $b \not\in L_0(vy)$ and $b' \not\in L_0(wx)$, as otherwise we can avoid the problem
by recoloring $vy$ with $b$ (the same as $wx$) or $wx$ with $b'$ (the same as $vy$).
  \item[(O4)] If $b \neq b'$ and $b' \in L(vw)$ then $c \not\in L_0(vy)$, as otherwise we can avoid the problem by
swapping the colors of $vy, vw$ from $b', c$ to $c, b'$.
  \item[(O5)] In $\widehat{R}_2$, $L(yz) = \{p, q\}$ where, for each $e \in \{vy,wy\}$, $p \in L(e) \setminus L_0(e)$, $q \in L_0(e)$, and
$L(e) = L_0(e) \cup \{p\}$.
\end{description}

We now consider four cases.

\textbf{Case 1:} $a = a'$ and $b = b'$. Then $L(uv) = L(uw) = L(vw) = \{a, b, c, d\}$, and $L_0(e) = \{a, b\}$
for every edge $e \in E(F)$ by (O1) and (O2). This is impossible in $\widehat{R}_1$ (where $|L_0(wy)| = 3$),
and it can be avoided in $\widehat{R}_2$ by changing the color of $yz$. In the latter case, by (O5),
$L(yz) = \{a, p\}$ or $\{b, p\}$, where $p \not\in \{a, b\}$. If $L(yz) = \{a, p\}$ then we can swap the colors of
$wy, yz$ from $a, p$ to $p, a$ and (re)color $uv, vw,wu$ with $c, d, a$ if $p = c$ and with $d, c, a$ otherwise;
the case $L(yz) = \{b, p\}$ is similar.

\textbf{Case 2:} $a = a'$ and $b \not= b'$. Then $L(vw) \subset \{a, b, b', c, d\}$. Thus $c \not\in L_0(vy)$ by (O2) and (O4),
and $b \not\in L_0(vy)$ by (O3), and so $L_0(vy) = \{a, b'\}$. Similarly, $L_0(wx) = \{a, b\}$. Now we can
avoid the problem by recoloring $vy$ and $wx$ with $a$, and then recoloring $vx$ and $wy$, unless
$L_0(vx) = L_0(vy) = \{a, b'\}$ and $L_0(wx) = L_0(wy) = \{a, b\}$. This is impossible in $\widehat{R}_1$, and it
can be avoided in $\widehat{R}_2$ by changing the color of $yz$. In the latter case, by (O5), $L(yz) = \{a, p\}$
for some color $p \not\in \{b, b'\}$ (since $p \not\in L_0(vy)\cup L_0(wy)$), and we can swap the colors of $wy$ and
$yz$ and (re)color $uv, vw,wu$ exactly as described at the end of Case 1.

\textbf{Case 3:} $a \not= a'$ and $b = b'$. This is equivalent to Case 2. (Interchange $v$ and $w$.)

\textbf{Case 4:} $a \not= a'$ and $b \not= b'$. We may assume that no color occurs in the lists $L_0$ of two
opposite edges of $F$, as otherwise we could use it on those two edges and get a new coloring
that is covered by a previous case. Thus each of the colors $a, a', b, b', c$ occurs in the list $L_0$
of either at most one, or two adjacent, edges of $F$.

\vspace{2mm}The proof of (b) is now easily completed. If (i) holds, then we can choose the color of
$yz$ so that some color occurs in the lists $L_0$ of two opposite edges of $F$, and this coloring
can be extended to $R_0$ by Theorem \ref{R-0-1-2-3}; as just remarked, this is covered by a previous case.
If (ii) holds, then we can avoid the problem by swapping the colors of $vy$ and $wy$, thereby
recoloring $vy$ with $a'\not\in L(uv)$ and $wy$ with $b' \not\in L(uw)$. So from now on we assume the graph
is $\widehat{R}_1$.

We claim that c does not occur in the list of any edge of $F$. Suppose it does, say
$c \in L_0(vy) = L(vy)$. (The other cases are similar.) By (O2) and (O4), this implies that
$\{b', d\} \cap L(vw) = \emptyset$ and so $L(vw) = \{a, a', b, c\}$; thus, by the analogs of (O4) for other edges
of $F$, $c \not\in L_0(vx)$ since $a \in L(vw)$, $c \not\in L_0(wx)$ since $b\in L(vw)$, and $c \not\in L_0(wy)$ since
$a' \in L_0(vw)$. It follows that $b'$ is in the list of another edge of $F$, as otherwise the lists of
$vx, wx$ and $wy$ contain the colors $a, a'$ and $b$ twice each, and some color is in the lists of two
opposite edges. So we can color $vy$ and $uw$ with $c$, some other edge of $F$ with $b'$, and then
complete the coloring using two different colors from $\{a, a', b\}$ on $F$ and the third color on
$vw$; and now $uv$ can be colored as it has two adjacent edges with the same color, $c$.

It follows that $L_0(e) \subset \{a, a', b, b'\}$ for every edge $e$ of $F$ in $\widehat{R}_1$, and so one of these colors
must occur on at least three edges, which contradicts the first paragraph of Case 4. This
completes the proof of Theorem \ref{R-1-2-bar}.
\end{proof}

\begin{lem}\label{color-triangle}
Let $T(syz)$ be a triangle with an associated function $L$ that assigns lists of
two colors to edges $sz$ and $yz$ and a list of four colors to $sy$. Then there are distinct colors
$a, b, c, d$ such that $L(yz) = \{a, c\}$, $\{b, d\} \in L(sy)$, and $T(syz)$ can be $L$-colored with the two
edges at $y$ receiving any of the three pairs of colors $\{a, b\}$, $\{b, c\}$ and $\{c, d\}$.
\end{lem}

\begin{proof}
Let $L(yz) = \{a, c\}$ and let $b, d$ be two colors in $L(sy) \setminus \{a, c\}$, labeled so that
$L(sz) \neq \{a, b\}, \{b, c\}$ or $\{c, d\}$. (In other words, if $L(sz)$ comprises one of $a, c$ and one of $b, d$,
then swap the labels of the colors so that $L(sz) = \{a, d\}$.) Then the result clearly holds.
\end{proof}

\begin{lem}\label{graphbelow}
Let $p$ be an arbitrary but fixed color. For each of the following conditions,
there is an $L$-coloring of the edges of $\widehat{R}_7$ below $xy$ in Figure \ref{more} such that the condition holds:\\
(C1) $xx'$ and $yy'$ have different colors;\\
(C2) $xx'$ does not have color $p$;\\
(C3) $yy'$ does not have color $p$.
\end{lem}

\begin{proof}
Choose colors $a, b, c, d$ that satisfy the conclusion of Lemma \ref{color-triangle} for triangle $T(sy'z)$
in $\widehat{R}_7$. The conclusion of Lemma \ref{graphbelow} clearly holds for (C2), as we can color $xx'$ differently
from $p$ and continue the coloring in order along the walk $xx'y'zsy'y$. The same holds for
(C1) if $b$ or $c \in L(xx')$, as then we can color $xx'$ and $sy'$ with $b$, or $xx'$ and $y'z$ with $c$, and
then color the remaining edges in order along the above walk. So in proving (C1) we may
assume that $L(xx') \cap \{b, c\} = \emptyset$. We must prove (C1) and (C3). There are two cases.

\textbf{Case 1:} $L(x'y') \not= \{b, c\}$. Color $x'y'$ with $q \not\in \{b, c\}$, then color $xx'$, and let $p$ denote the
color of $xx'$ if we are proving (C1). Since $q$ cannot equal both $a$ and $d$, assume w.l.o.g. $q \not= a$,
and color $y'z, sy'$ with $a, b$ if $L(yy¡ä) = \{b, c, p, q\}$ and with $c, b$ otherwise; then the coloring
can be completed with $yy'$ not receiving color $p$.

\textbf{Case 2:} $L(x'y') = \{b, c\}$. In proving (C1), give $xx'$ a color and call it $p$, noting that
$p \not\in \{b, c\}$ by the assumption made before Case 1. Color $x'y'$, $y'z$, $sy'$ with $b, c, d$ if $L(yy') =
\{a, b, c, p\}$ and with $c, a, b$ otherwise. Then the coloring can be completed with $yy'$ not
receiving color $p$.
\end{proof}

\begin{lem}\label{kite-lemma}
Let $T(vwx)$ be a triangle and $vy$ be an edge incident with it.
If there is an associated function $L$ that assigns lists of two colors to edges $vy,wx$ and $vx$ and a list of four colors to $vw$ so that $L(vy)\cap L(wx)=\emptyset$, then we can $L$-color $vy,wx$ and $vx$ so that there are at least two colors still available for $L$-coloring $vw$ while extending this partial coloring if (i) $L(vw)\neq L(vy)\cup L(wx)$, or (ii) $L(vx)\not\subset L(vw)$.
\end{lem}

\begin{proof}
  Suppose that $L(vy)=\{a,b\}$ and $L(wx)=\{c,d\}$.

  (i) Since $L(vw)\neq \{a,b,c,d\}$, either $\{a,c\}\not\subset L(vw)$ or $\{b,d\}\not\subset L(vw)$. Assume, w.l.o.g., that $\{a,c\}\not\subset L(vw)$. If $L(vx)\neq \{a,c\}$, then color $vy,wx$ with $a,c$, and color $vx$ with a color in $L(vx)\setminus \{a,c\}$. If $L(vx)=\{a,c\}$, then color $vy,wx,vx$ with $b,c,a$. In any of the above two cases, there are at least two colors still available for $L$-coloring $vw$.

  (ii) $L$-color $vx$ with a color not in $L(vw)$, and $vy, wx$ can then be easily colored. Obviously, there are at least two colors still available for $L$-coloring $vw$ at this stage.
\end{proof}

\begin{thm}\label{R-4-5-6-7-8-9-bar}
Each of the graphs $R_4$ and $\widehat{R}_4,\ldots,\widehat{R}_{9}$ is $L$-colorable.
\end{thm}

\begin{proof}
Obviously, if $\widehat{R}_4$ is $L$-colorable then so is $R_4$, and so we will consider just the last
four graphs.

\textbf{Case 1:} $\widehat{R}_4$. By Lemma \ref{color-triangle}, there are distinct colors $a, b, c, d$ such that the edges of the
triangle $syz$ can be $L$-colored with the two edges at $y$ receiving any of the three pairs of
colors $\{a, b\}, \{b, c\}$ and $\{c, d\}$. If $L(wy) \not= \{a, b, c, d\}$ then at least one of these colorings will
leave $wy$ with at least three available colors, and the result will hold since $\widehat{R}_1$ is $L$-colorable
by Theorem \ref{R-1-2-bar}(a). Thus we may assume that $L(wy) = \{a, b, c, d\}$. By symmetry, we may
also assume that $L(vy) = \{a, b, c, d\}$. Now $L$-color $yz$ with color $c$, so that there are at least
two colors ($b$ and $d$) available for $sy$ that enable $sz$ to be colored as well. The coloring can
now be completed by Theorem \ref{R-1-2-bar}(b)(ii) applied to $\widehat{R}_4 \setminus \{yz, sz\}$, which is essentially the
same as $\widehat{R}_2$.

\textbf{Case 2:} $\widehat{R}_5$. If $L(xy)$ contains a color that is not in $L(e)$ for some $e \in E(F)$, color $xy$ with
such a color and then color $yz$; now $e$ still has three available colors, and so the remaining
edges can be colored since $\widehat{R}_1$ is L-colorable by Theorem \ref{R-1-2-bar}(a). So we may assume that
$L(xy) \subseteq L(e)$ for every $e \in E(F)$. If we now color $xy$ with a color from $L(xy) \setminus L(yz)$, then
the remaining two colors from $L(xy)$ are still available for every edge of $F$, and the coloring
can be completed by Theorem \ref{R-1-2-bar}(b)(i).

\textbf{Case 3:} $\widehat{R}_6$ and $\widehat{R}_7$. To enable these graphs to be discussed together, let $x'$ and $y'$ both
denote the vertex $z$ in $\widehat{R}_6$, and note that the conclusion of Lemma \ref{graphbelow} easily holds for $\widehat{R}_6$.

Suppose first that, for some edge $e \in E(F)$, $L(xy)$ contains a color $p \not\in L(e)$. If $e$ is
incident with $y$, $L$-color the edges below $xy$ so that (C2) of Lemma \ref{graphbelow} holds, and then color
$xy$ so that $p$ is used at $y$ (i.e., color $xy$ with $p$ if $yy'$ is not already colored with $p$); if $e$ is
incident with $x$, do the same with $x$ and $y$ interchanged, using (C3). Then the edge $e$ still
has at least three available colors, and so the remaining edges can be colored since $\widehat{R}_1$ is
$L$-colorable by Theorem \ref{R-1-2-bar}. We deduce from this that $L(e) = L(xy) = \{a, b, c, d\}$, say, for
every edge $e \in E(F)$.

Now color the edges below $xy$ so that (C1) holds, say $xx'$ and $yy'$ have different colors $a'$
and $b'$ respectively, and relabel $a, b, c, d$ if necessary so that $a'\not\in \{b, c, d\}$ and $b' \not\in \{a, c, d\}$.
Then the coloring can be completed as follows. Color $xy$ with $d$. If $a \not\in L(uv)$ then color the
edges $vx, xw,wy, yv$ of $F$ with $c, b, c, a$ and color the remaining edges in the order $vw, uw, uv$;
otherwise, if $a \in L(uv)$, color the edges of $F$ with $b, c, a, c$, color uv with $a$, and then color
$vw$ and $uw$.

\textbf{Case 4:} $\widehat{R}_8$. If possible, $L$-color $wy'$ with a color in $L(wy')\setminus L(uw)$, and then it is easy to color the remaining edges in this order $yy',vy,wx,xx',vw,vx',uv,uw$. So we assume that $L(uw)=L(wy')=\{a,b,c,d\}$. By symmetry, we also assume that $L(uv)=L(vx')$.

If possible, $L$-color $wy'$ with a color in $L(wy')\setminus L(vw)$, and color in this order $yy',vy,wx,xx'$; the remaining edges can then be colored since $T_2$ is $L_2$-colorable as in Theorem \ref{six-useful-graphs}(b). So we assume that $L(vw)=L(wy')=\{a,b,c,d\}$, and by symmetry, assume that $L(vx')=L(vw)=\{a,b,c,d\}$. This implies that $L(uv)=\{a,b,c,d\}$.

If possible, $L$-color $vy$ with a color in $L(vy)\setminus L(vw)$, and continue the coloring in order along the walk $yy'wxx'$; the remaining edges can then be colored since $T_2$ is $L_2$-colorable. So we assume that $L(vy)\subset L(vw)$, and by symmetry, assume that $L(wx)\subset L(vw)$.

If possible, $L$-color $vy$ and $wx$ with a same color, and color in this order $yy',wy',xx'$; the remaining edges can then be colored since $T_2$ is $L_2$-colorable. So we assume, w.l.o.g., that $L(vy)=\{a,b\}$ and $L(wx)=\{c,d\}$.

Since either $L(yy')\neq \{a,d\}$ or $L(yy')\neq \{b,d\}$, we assume, w.l.o.g., that $L(yy')\neq \{a,d\}$. Now color $uv,uw,vw,vx',vy,wx,wy'$ with $d,a,b,c,a,c,d$, respectively, and then $xx'$ and $yy'$ can be easily colored.

\textbf{Case 5:} $\widehat{R}_9$. Remove the edge $wy'$ and add a new edge $wy$ so that $L(wy)=L(wy')$. It is easy to see that the resulting graph $\widehat{R}^*_9$ is $L$-colorable only if $\widehat{R}_9$ is $L$-colorable.

 We now consider $\widehat{R}^*_9$. If possible, $L$-color $yy'$ with a color in $L(yy')\setminus L(vy)$; the remaining edges can then be colored since $\widehat{R}_1$ is $L$-colorable by Theorem \ref{R-1-2-bar}(a). So we assume that $L(yy')=L(vy)=\{a,b\}$.

In the following we consider $\widehat{R}_9$. If possible, $L$-color $vy$ and $wx$ with a same color, and color in this order $yy',vx$; the remaining edges can then be colored since $T_2$ is $L_2$-colorable. So we assume that $L(vy)=\{c,d\}$.

If $L(vw)\neq L(vy)\cup L(wx)$ or $L(vx)\not\subset L(vw)$, then by Lemma \ref{kite-lemma} we can $L$-color $vy,wx,vx$ so that there are at least two colors still available for $L$-coloring $vw$. Now it is easy to color $yy'$ and the remaining edges can be $L$-colored since $T_2$ is $L_2$-colorable. So we assume that $L(vx)\subset L(vw)= L(vy)\cup L(wx)=\{a,b,c,d\}$.

Since either $L(vx)\neq \{a,c\}$ or $L(vx)\neq \{b,c\}$, we assume, w.l.o.g., that $L(vx)\neq \{a,c\}$. Color $vy,yy',wx$ with $a,b,c$, and $L$-color $vx,vw,wy'$ with $p\not\in \{a,c\}$, $q\not\in \{a,c,p\}$, $r\not\in  \{b,c,q\}$, respectively. Note that $p\in L(vx)\subset\{a,b,c,d\}$, $q\in L(vw)=\{a,b,c,d\}$ and $q\neq p$. Thus
$\{p,q\}=\{b,d\}$.

We now finish the $L$-coloring of $\widehat{R}_9$ by coloring $uv$ and $uw$ properly. The only problem is if there is a color $s$ so that $L(uv)=\{a,p,q,s\}$ and $L(uw)=\{c,q,r,s\}$.

If $|\{b,c,q\}|=2$, then erase the color on $wy'$ and $L$-color $wy'$ with a color $r'\not\in  \{b,c,q,r\}$. Hence we can finish the $L$-coloring of $\widehat{R}_9$ by coloring $uv,uw$ with $s,r$. So we assume that $|\{b,c,q\}|=3$, which implies that $q\neq b$, and then $q=d,p=b$.

Under this condition, we exchange the colors on $vw$ and $wx$, and then color $uv,uw$ with $d,s$. This completes the proof since we get an $L$-coloring of $\widehat{R}_9$.
\end{proof}

\section{Proof of the main results}

\begin{thm}\label{core-theorem}
Let $H$ be a graph and let $L(e)$ be a list of four colors, for each $e\in E(H)$. If $H$ is not edge $L$-colorable but every proper subgraph of $H$ has an edge $L$-coloring, then $H$ does not contain any of the configurations $G_1,\ldots,G_{14}$, $S_1,S_2$ or $S_3$
\end{thm}

\begin{proof}
We consider each configuration $C$ separately. We let $H¡ä$ be the subgraph of $H$ obtained by removing all the solid vertices in $C$. We choose an edge $L$-coloring of $H¡ä$, which
exists by hypothesis, and for each edge $e \in E(H) \setminus E(H¡ä)$ we denote by $A(e)$ the set of available colors for $e$, comprising those colors in $L(e)$ that are not used on any colored edge adjacent to $e$. If we can prove that $H-E(H¡ä)$ is edge $A$-colorable, then it will follow that $H$ is edge $L$-colorable, and this contradiction will show that $H$ cannot contain the configuration $C$. In each case, it suffices to describe how to construct an edge $A$-coloring of $H-E(H¡ä)$.

In every case except for $G_1$ and $G_3$, $H-E(H¡ä)$ is isomorphic to a graph that has already been proved to be $A$-colorable for a suitable list assignment $A$, as shown in this table.

\begin{center}
\begin{tabular}{ c c c c c c c c c c c c c c c}
  Configuration & $G_2$ & $G_4$ & $G_5$ & $G_6$ & $G_7$ & $G_8$ & $G_9$ & $G_{10}$ & $G_{12}$ & $G_{13}$ & $G_{14}$ & $S_1$ & $S_2$ & $S_3$ \\
  Equivalent graph & $T_3$ & $T_4$ & $T_6$ & $T_5$ & $\widehat{R}_8$ & $\widehat{R}_9$ & $\widehat{R}_1$ & $R_1$ & $R_3$ & $\widehat{R}_4$ & $R_4$ & $\widehat{R}_5$ & $\widehat{R}_6$ & $\widehat{R}_7$ \\
  Theorem & \ref{six-useful-graphs} & \ref{six-useful-graphs} & \ref{six-useful-graphs} & \ref{six-useful-graphs} & \ref{R-4-5-6-7-8-9-bar} & \ref{R-4-5-6-7-8-9-bar} & \ref{R-1-2-bar} & \ref{R-0-1-2-3} & \ref{R-0-1-2-3} & \ref{R-4-5-6-7-8-9-bar} & \ref{R-4-5-6-7-8-9-bar} & \ref{R-4-5-6-7-8-9-bar} & \ref{R-4-5-6-7-8-9-bar} & \ref{R-4-5-6-7-8-9-bar} \\
\end{tabular}
\end{center}

It remains to consider the graphs $G_1$ and $G_3$. In $G_1$, let $u$ denote the solid vertex and let $x$ denote the left vertex in that picture. It is easy to see that $|A(ux)|\geq 1$ and $|A(uy)|\geq 2$, and so we can color the edges in the order $ux,uy$. In $G_3$, the result holds since every edge of the 4-cycle has at least two available colors and it is well-known that a 4-cycle is edge 2-choosable. This proves Theorem \ref{core-theorem}.
\end{proof}

\proof[\textbf{Proofs of Theorems \ref{o1p} and \ref{3}}] Let $H$ be the minimum counterexample to the theorem and let $L(e)$ be a list of four colors for each $e\in E(H)$.
First of all, it is easy to see that $\delta(H)\geq 2$. Since every proper subgraph of $H$ has an edge $L$-coloring, $H$ does not contain any of the configurations listed in Theorem \ref{core-theorem}.
On the other hand, if $\Delta(H)=4$ and $\vartheta(H)\geq 3$, then $H$ contains one of the configurations $G_1,\ldots,G_{14}$, $S_1,S_2$ or $S_3$ by Theorem \ref{mainthm-2}, contradicting Theorem \ref{core-theorem}. If $\Delta(G)=3$, then $H$ contains the configuration $G_1$ or $G_3$ or $G_{10}$, which is an immediate corollary from the result of Zhang, Liu and Wu \cite[Theorem 4.2]{ZPOPG}. However, this also contradicts Theorem \ref{core-theorem}, which implies that $H$ does not contain $G_1$ or $G_3$ or $G_{10}$. $\hfill\square$

%


\begin{thebibliography}{10}\setlength{\itemsep}{-3pt}

\bibitem{Alon} N. Alon, Combinatorial Nullstellensatz, \emph{Combin. Probab. Comput.} 8 (1999) 7--29.

\bibitem{Alon1996} N. Alon, M.B. Nathanson, I.Z. Ruzsa, The polynomial method and restricted sums of congruence classes,
\emph{J. Number Theory} 56 (1996), 404--417.


\bibitem{Borodin.lcc} O.V. Borodin, A.V. Kostochka, D. R. Woodall, List edge and list total colourings of multigraphs, {\it J. Combin. Theory Ser. B} 71 (1997) 184-204.


\bibitem{Eggleton} R.B. Eggleton, Rectilinear drawings of graphs, \emph{Utilitas Math.} 29 (1986) 149--172.

\bibitem{HW} T.J. Hetherington, D. R. Woodall, Edge and total choosability of near-outerplanar graphs, \emph{Electron. J. Combin.} 13 (2006) \#R98.

\bibitem{JT-book}  T.R. Jensen, B. Toft, \emph{Graph Coloring Problems}, New York: Wiley-Interscience, 1995.

\bibitem{JMT} M. Juvan, B. Mohar, R. Thomas, List edge-colorings of series-parallel graphs, \emph{Electron. J. Combin.} 6 (1999) \#R42.




\bibitem{Kral} D. K\'ral, L. Stacho, Coloring plane graphs with independent crossings, \emph{J. Graph Theory} 64(3) (2010) 184--205.

\bibitem{Lewis} R. Lewis, \emph{A Guide to Graph Colouring: Algorithms and Applications}, Springer International Publishers, 2015.



\bibitem{TZ} J. Tian, X. Zhang, Pseudo-outerplanar graphs and chromatic conjectures, \emph{Ars Combin.} 114 (2014) 353--361.



\bibitem{Wang2} W. Wang, K.-W. Lih, Choosability, edge-choosability and total choosability of outerplane graphs, \emph{European J. Combin.} 22 (2001) 71--78

\bibitem{Wang} W. Wang, K. Zhang, $\Delta$-Matchings and edge-face chromatic numbers. \emph{Acta Math. Appl. Sinica} 22 (1999) 236--242.

\bibitem{LTC} X. Zhang. List total coloring of pseodo-outerplanar graphs, \emph{Discrete Math.} 313 (2013) 2297--2306.

\bibitem{ZPOPG} X. Zhang, G. Liu, J. L. Wu. Edge covering pseudo-outerplanar graphs with forests, \emph{Discrete Math.} 312 (2012) 2788--2799.



\bibitem{total} X. Zhang. Total coloring of outer-1-planar graphs with near-independent crossings, \emph{J. Comb. Optim.} 34(3) (2017) 661--675.

\bibitem{edge} X. Zhang. The edge chromatic number of outer-1-planar graphs, \emph{Discrete Math.} 339 (2016) 1393--1399.

\end{thebibliography}
\end{document}